\documentclass[11pt]{amsart}

\usepackage{amsmath,amssymb, amsfonts, amsbsy, latexsym, color,bbm}%,amsthm}
\usepackage{graphicx}
\usepackage{hyperref}

\newtheorem{thm}{Theorem}[section]

\newtheorem{cor}[thm]{Corollary}

\newtheorem{lemma}[thm]{Lemma}
\newtheorem{lem}[thm]{Lemma}
\newtheorem{prop}[thm]{Proposition}

\newcommand\newnumbered[2]{\newtheorem{#1}{#2}}
\newnumbered{defn}{Definition}
\newnumbered{rem}{Remark}
\newnumbered{conclusion}{Conclusion}
\newnumbered{remarks}{Remarks}
\newnumbered{ex}{Example}
\newnumbered{exc}{Exercise}
\newnumbered{facts}{Facts}
\newnumbered{prob}{Problem}
\newnumbered{question}{Question}
\newnumbered{answer}{Answer}
\newnumbered{conj}{Conjecture}
\numberwithin{equation}{section}

\newcommand{\beq}{\begin{equation}}
\newcommand{\eeq}{\end{equation}}

\begin{document}

\title[A scattering transform for graphs based on heat semigroups]{A scattering transform for graphs based on heat semigroups, with an application for the detection of anomalies in positive time series with underlying periodicities}%

\author{Bernhard G.\ Bodmann and Iris Emilsdottir}
\email[B. G. Bodmann]{bgb@math.uh.edu}

   % \dedication{In memory of John Haas}
    
  %  \classno{81P45 (primary), 05B05, 42C15, 52A40, 52C35 (secondary).}
 
\thanks{This work was supported by the National Science Foundation grant ATD-1925352. The group from the MIT Lincoln Labs is gratefully acknowledged for the challenge problem in the 2019 ATD program meeting motivating the research presented here.}

    \begin{abstract}
    This paper develops an adaptive version of Mallat's scattering transform for signals on graphs.
    The main results are norm bounds for the layers of the transform, obtained from a version of
    a Beurling-Deny inequality that permits to remove the nonlinear steps in the scattering transform. 
   Under statistical assumptions on the input signal, the norm bounds can be refined. The concepts presented here are illustrated with an application to traffic counts which exhibit characteristic daily and weekly periodicities. Anomalous traffic patterns which deviate from
   these expected periodicities produce a response in the scattering transform.
       \end{abstract}

\maketitle

\section{Introduction}

The results discussed here demonstrate features of a scattering transform for signals on graphs. 
The ideas we develop are borrowed from Mallat's construction of a scattering transform for
signals on Euclidean spaces, whose
main purpose is to serve as a \emph{non-adaptive} method that generates feature
vectors from signals for classification purposes \cite{mallat2012group,sifre2013rotation,
bruna13-scattering,anden2014deep,oyallon2015deep}. The use of standard methods such as support vector machines
on these feature vectors is successful because the output of the scattering transform behaves well under smooth deformations of the signal.

Signals on graphs have recently received growing attention in the literature \cite{henaff15-deepgraph,ShuRV16,defferrard17-cnngraphs,HecCQ17}, because many types of data
are not regularly structured. For example, climate data is acquired at irregularly spaced weather stations around the globe, 
traffic is measured at distinct points in a street map, and  internet traffic may be viewed as
data being transmitted between nodes that are scattered across the world. 

When analyzing signals on graphs, standard techniques from time-\-fre\-quen\-cy analysis are not as
immediately available as in the case of signals on Euclidean spaces. Chen and Mallat developed a scattering transform
on graphs based on Haar wavelets, see \cite{chen2014unsupervised}, but an implementation of this structure with 
more regularity needs tools from
graph signal processing, in particular wavelets on graphs \cite{coifman2006diffusionmaps,coifman2006diffusion}.
Perhaps the closest design of a scattering transform for signals on graphs was presented by Zou and Lerman \cite{giladpaper},
who showed properties that mimic the behavior of Mallat's construction on Euclidean spaces.

In this paper, we also use ideas from graph signal processing \cite{hammond2009wavelets,ShuRV16} to develop a scattering transform, whereby
the eigenvalues of a Laplacian play the role of (squared) frequencies,
and the eigenvectors of the Laplacian then give rise to a type of Fourier transform. This permits us
to develop high and low-pass filters that are somewhat analogous to Mallat's convolution with a wavelet.
Here, we are content to have a pair of high and low pass filters based on the heat kernel 
generated by a modified graph Laplacian. A main difference to the work by Zou and Lerman and to the papers
of Mallat and co-authors is that we exploit the freedom to choose the Laplacian in a signal-adaptive way.
This means, the resulting transform can learn and adapt to signal characteristics. With additional freedom in the design of the transform,
one may hope to get closer to the level of performance demonstrated by less structured neural networks 
\cite{lecun15-deeplearning,henaff15-deepgraph,bronstein17-geomdeeplearn,gama18-gnnarchit}

The use of the heat kernel rests on a special property that it confers on the transform presented here; it
is positivity preserving \cite{simon,bratteli}, which is instrumental for
the derivation of norm bounds in the scattering transform. More precisely, we can show
that the norm of a signal decays exponentially as the layer increases. This is analogous to a result derived by
Waldspurger \cite{waldspurger} for Mallat's scattering transform based on analytic wavelets, and a bound
by Zou and Lerman for graph scattering \cite{giladpaper}. However, in contrast to this latter result, which 
depends explicitly on the size of the graph and deteriorates as the graph grows, our norm bounds only depend 
on the graph geometry through
the spectrum of the graph Laplacian.

%In contrast to Mallat's theory, the scattering transform developed here adapts to the signals.
The adaptivity of the scattering transform presented here is intended as a step towards provable performance guarantees. Indeed, in the context of statistics,
with model assumptions on a class of random graph signals, we can derive guarantees on the behavior of
the scattering transform as a goodness-of-fit test for noisy data on graphs.

The motivation for the type of scattering transform we develop is the problem of finding anomalous behavior in time series with
typical underlying periodicities. 
While it may seem that the structure of time series is simply given by the temporal order,
periodicities can motivate organizing the data in a more complex way, for example making 
points in time adjacent when they differ by typical periods for which periodicities in the data are expected. In the concrete application discussed here, traffic counts are evaluated. These counts can be expected to show daily, weekly and yearly periodicities that could be included in
models and analysis. We use a stochastic model to estimate parameters governing the traffic counts. This model can be described as a multiplicative random walk with log-normal distributions. Once the parameters are estimated, we generate the output of the scattering transform and
evaluate the goodness of fit. The application is described and illustrated with plots and images.

The main contributions of this paper are the following: 
\begin{itemize}
\item We develop a scattering transform based on heat semigroups whose generator can be chosen within a family of Laplacians that depend on the data,
thus the transform is adaptive, but retains similar properties as the transforms described by Zou and Lerman \cite{giladpaper}.
\item We establish that the norm for the output of the $k$-th layer of the scattering transform decays exponentially in $k$. The key to
unlocking deep network behavior is a version of a Beurling-Deny inequality \cite{simon}. The norm bounds are refined with statistical assumptions on the input signal. If the input is a multiplicative random walk on a graph with a log-normal distribution at each vertex, then
the norm of the output concentrates for each layer. The result on norm decay depending on characteristics of the input signal
is similar to the objective of the study of Mallat's scattering transform by Wiatowski, Grohs, and B{\"o}lcskei \cite{WiaGB17,
wiatowski2018mathematical}, where
regularity assumptions lead to refined norm bounds in comparison with Waldspurger's general estimates.
\item The theoretical concepts are illustrated with an application to traffic counts \cite{CalTrans} and anomaly detection.
\end{itemize}

The remainder of this  paper is structured as follows: 
In the next section, we introduce terminology and fix notation. The scattering transform is introduced in Section~\ref{sec:scatter}.
We establish norm bounds similar to known results for the non-adaptive scattering transform.
In Section~\ref{sec:stat}, we refine the norm bounds with assumptions on the signal, given by a stochastic process that has similarities with geometric Brownian motion. Finally, we illustrate the concepts with an application to traffic counts in Section~\ref{sec:traffic}

\section{Preliminaries}\label{sec:prelim}

%\begin{defn}
We begin by fixing terminology and notation.
A \emph{graph} $\Gamma = (V,E)$ is given by a finite vertex set $V$ and an edge set $E$ containing unordered pairs of $V$, subsets containing two vertices. This is also referred to as a simple, undirected finite graph in the literature.
An \emph{oriented graph} is described by a vertex set $V$ and an edge set $E$, for which $E$ contains \emph{ordered} pairs 
of vertices. When \emph{choosing an orientation} for a graph, we replace each element $\{i,j\} \in E$ by an ordered pair $(i,j)$.
The ordering can be arbitrary or determined by additional structural input, such as an ordering relation on the vertices.
When considering a directed graph, we also speak of an \emph{edge without orientation} when passing from $(i,j) \in E$ to $\{i,j\}$.  
Two edges are \emph{adjacent} if they have a vertex in common. 
A graph is \emph{connected} if any two vertices in $V$ appear in a sequence of vertices such that each pair of consecutive
elements in this sequence forms an edge. A directed graph is \emph{weakly connected}
if any two vertices appear in a sequence of adjacent edges without orientation.
%When \emph{reversing the direction} of an edge $e=(i,j)$, we write $\tilde e = (j,i)$. The \emph{edge set with reversed directions} is
%denoted by $\tilde E$.
 
The Hilbert space $\ell^2(V)$ is the space of all real-valued functions $f: V \to \mathbb R$, equipped with the canonical 
inner product that associates $\langle f, g\rangle = \sum_{j \in V} f(j) g(j)$ with $f, g \in \ell^2(V)$. The norm induced
by the inner product is written as usual, $\|f\| = (\langle f,f\rangle)^{1/2}$.
The \emph{graph Laplacian} $\Delta$ is the self-adjoint operator corresponding to the quadratic form defined by 
$Q(f) = \sum_{\{i,j\} \in E} |f(i)-f(j)|^2$ for $f \in \ell^2(V)$. 

In the following, we modify the (standard) graph Laplacian with additional elements that permit us to adapt it to given data.
The main insight is that key properties of the Laplacian remain intact. In particular, with the modification we choose, 
the functions in the kernel of the Laplacian for a weakly connected, directed graph form a one-dimensional subspace as usual.

Given a directed graph $(V,E)$ and two functions $w : E \to \mathbb{R}^+\setminus\{0\}$ and $a: E \to \mathbb R$, we let $\Delta_{w,a}$
be the operator on $\ell^2(V)$ corresponding to the quadratic form $f \mapsto Q_{w,a}(f) = \sum_{(i,j) \in E} w(i,j) |e^{a(i,j)} f(i)-f(j)|^2 $.  In this context, we call the function $w$ a \emph{weight} on the edges and $a$ a \emph{drift}.
%For two functions $w, a$ on the edge set of a directed graph $(V,E)$ with $w$ assuming only strictly positive values, we let $\Delta_{w,a}$ be defined via the quadratic form
%$Q_{w,a}(f) = \sum_{(i,j) \in E} w(i,j) |e^{a(i,j)} f(i)-f(j)|^2 $. 
%We then say that the edges are weighted by $w$. 
If there is $\phi: V \to \mathbb R$
such that for each $(i,j) \in E$, $a(i,j) = \phi(j)-\phi(i)$, then we say that $a$ is a \emph{gradient field} and $\phi$ is  a \emph{potential}. 
For the following results, we recall that all graphs and directed graphs considered here are finite, so any function $f: V \to \mathbb R$ is in
$\ell^2(V)$.%\end{defn}

\begin{prop} If $(V,E)$ is a directed graph, $w: E \to \mathbb R^+\setminus\{0\}$ a weight function and $a: E \to \mathbb R$ a gradient field, and
the graph is weakly connected, then $\Delta_{w,a}$ has a one-dimensional kernel.
\end{prop}
\begin{proof}
By assumption, $a$ is a gradient field, so there is $\phi: V \to \mathbb{R}$ and for each $(i,j) \in E$, $e^{a(i,j)} = e^{\phi(j)}/e^{\phi(i)}$.
If $f$ is annihilated by $\Delta_{w,a}$, then for each $(i,j) \in E$, $e^{\phi(j)} f(i) = e^{\phi(i)} f(j)$. This determines the value of $f(j)$ given $f(i)$. Starting from the value $f(i_0)$
at any fixed vertex $i_0$, based on the connectedness assumption, we can then solve for $f(j) = f(i_0) e^{-\phi(i_0)} e^{\phi(j)}$ at any $j \in V$. 
This shows that $\Delta_{w,a}$ has a one-dimensional kernel.
\end{proof}

Using the Courant-Fischer variational characterization of eigenvalues for Hermitian operators \cite{HornJohnson}, we conclude with a consequence for the smallest non-zero eigenvalue of $\Delta_{w,a}$, obtained by minimizing the quadratic form $Q_{w,a}$ subject to an orthogonality constraint.

\begin{cor}
Let $(V,E)$ be a directed, weakly connected graph, $w: E \to \mathbb R^+\setminus\{0\}$ a weight function and $a: E \to \mathbb R$ a gradient field with potential $\phi: V \to \mathbb R$, then the smallest non-zero eigenvalue of $\Delta_{w,a}$ is given by
$$
   \lambda_1(\Delta_{w,a}) = \min \{ \langle \Delta_{w,a} f, f\rangle: f \in \ell^2(V), \|f\|=1, \sum_{i \in V} f(i) e^{\phi(i)} = 0\} \, .
$$
\end{cor}

Next, we study the properties of the contraction semigroup generated by the negative Laplacian $-\Delta_{w,a}$. We suppress the dependence on
$w$ and $a$ and define for $f \in \ell^2(V)$ and $t \ge 0$
$$
   G_t f  = e^{-t \Delta_{w,a}} f \, ,
$$
then $\{G_t \}_{t \ge 0}$ satisfies the semigroup property $G_{t} G_{t'} = G_{t+t'}$ for $t,t' \ge 0$.
Since $\Delta_{w,a}$ is Hermitian and only has non-negative eigenvalues, $\| G_t  f\| \le \|f\|$. We can say more
based on the concrete definition of $\Delta_{w,a}$. From the expression for the quadratic form, for any $f \in \ell^2(V)$,
the inequality
$\langle \Delta_{w,a} |f|, |f| \rangle \le \langle \Delta_{w,a} f, f \rangle$ holds, where $|f|(i)\equiv |f(i)|, i \in V$. This inequality is known to be equivalent to $\{G_t\}_{t \ge 0}$
being a positivity preserving semigroup \cite{bratteli,simon}, meaning that for $t \ge 0$ and any $f \in \ell^2(V)$ with $f(i) \ge 0$ for each $ i \in V$, we have
$ G_t f (i) \ge 0$ for each $i \in V$. Because this is a central property, we give a separate proof that the semigroup is positivity preserving.
To facilitate the proof, we introduce additional operators.
Let $M_w$ be the multiplication operator corresponding to $w$ on $\ell^2(E)$, so $M_w h(i,j) = w(i,j) h(i,j)$ for any $(i,j) \in E$, and let 
$$
   D_a: \ell^2(V) \to \ell^2(E), (D_a f)(i,j) = f(j)-e^{a(i,j)} f(i) \, .
$$
 This permits us to write the Laplacian as a composition of these operators,
$$
   \Delta_{w,a}  = D_a^* M_w D_a \, .
$$

\begin{prop}
Let $t \ge 0$, $(V,E)$ be a directed graph, $\Delta_{w,a}$ be the Laplacian associated with the functions 
$w: E \to \mathbb R^+\setminus\{0\}$  and 
$a: E \to \mathbb R$, then $\{e^{-t \Delta_{w,a}}\}_{t \ge 0}$ is positivity preserving.
\end{prop}
\begin{proof}
We have  for $i \in V$ and any $n \in \mathbb N$, $\langle (- \Delta_{w,a})^n \delta_i, \delta_i \rangle \ge - \|\Delta_{w,a}\|^n $ where $\|\Delta_{w,a}\|$ is the operator norm
of $\Delta_{w,a}$,  the largest eigenvalue of this Hermitian operator.
Using this estimate to get lower bounds for the terms in the power series expansion of $e^{-t \Delta_{w,a}}$ gives
$$
  \langle e^{-t \Delta_{w,a}} \delta_i , \delta_i \rangle \ge 1 - (e^{t \|\Delta_{w,a}\| }  - 1) \, ,
$$
which is seen to be positive for sufficiently small $t$, such as $t<\ln 2/\|\Delta_{w,a}\|$.
Moreover, if $\{i,j\} \subset V$,  then  either $(i,j) \in E$
and
$$
   \langle - \Delta_{w,a} \delta_i, \delta_j \rangle = w(i,j) e^{a(i,j)} 
$$
or $(i,j) \not \in E$ and then $\langle - \Delta_{w,a} \delta_i, \delta_j \rangle = 0$ .
Again, based on the power series expansion of $G_t = e^{-t \Delta_{w,a}}$,
we have $\frac{d}{dt}|_{t=0}  \langle G_t \delta_i, \delta_j \rangle = \langle - \Delta_{w,a} \delta_i, \delta_j \rangle > 0$
if $(i,j) \in E$, so 
 for all
sufficiently small $t$, $\langle G_t \delta_i, \delta_j\rangle \ge 0$. 
If $i$ and $j$ are at a distance $d(i,j)>1$,  then the first non-zero term in the power series
for $\langle G_t \delta_i , \delta_j\rangle $ is 
$(-t)^{d(i,j)} \langle \Delta_{w,a}^{d(i,j)} \delta_i , \delta_j\rangle $
and the only contribution comes from paths of length $d(i,j)$ between $i$ and $j$. Factoring this product into
 $ \sum_l -t \langle \Delta_{w,a} \delta_i, \delta_l \rangle (-t)^{d(i,j)-1} \langle \Delta_{w,a}^{(d(i,j)-1} \delta_l, \delta_j \rangle \ , $
 then the summation over $l$ only includes vertices with $(i,l) \in E$ and $l$ and $j$ having distance $d(i,j)-1$. Using induction then shows that the $d$-th power of the negative Laplacian has positive off-digonal entries $\langle (- \Delta_{w,a})^d \delta_i, \delta_j\rangle$
 if $d(i,j)= d$. Exhausting all pairs of vertices shows that the semigroup is positivity preserving for all sufficiently small $t$.
Now using the semigroup 
property, $G_t = G_{t/n}^n$ and choosing $n$ sufficiently large gives that $G_t$ is positivity preserving
for any $t \ge 0$.
\end{proof}
%\end{prop}

\section{A scattering transform based on heat semigroups}\label{sec:scatter}

In this section, we introduce and discuss properties of a scattering transform associated with $\Delta_{w,a}$, in the spirit of
Mallat's design of a scattering transform based on wavelets. In our setting of functions on graphs, there is typically no translation invariance
as in the construction of wavelets for $L^2(\mathbb{R})$. Instead, we use the heat semigroup generated by the Laplacian 
we introduced in order to define high and low-pass components of a function.
Following Mallat's idea, we combine this separation of frequency content with a
non-linear local operation, which is norm preserving. This second step helps to identify singular features in the function.
The main purpose of this definition is to illustrate the idea without ancillary bells and whistles. The typical form of the scattering 
transform would involve several frequency bands and possibly a rectifying linear unit instead of  the use of the absolute value here.
It is left to the reader to make this construction more sophisticated if needed.

\begin{defn}
Given $t>0$, weights $w: E \to \mathbb{R}^+\setminus \{0\}$ and $a: E \to \mathbb R$ and the associated Laplacian $\Delta_{w,a}$, 
we let
$$
   T_t = e^{-t \Delta_{w,a}/2}
$$
and 
$$
  S_t = (I- e^{-t \Delta_{w,a}})^{1/2} \, .
$$
For any $f \in \ell^2(V)$, we let
$ f_0=0 , g_0 = f$, $f_1 = T_t |g_0|$, $g_1 = S_t |g_0|$, and successively define for $k \in \mathbb N$
$f_{k+1} = T_t |g_k|$ as well as $g_{k+1} = S_t |g_k|$.
\end{defn}

We deduce simple properties of the sequence $(f_k,g_k)_{k=0}^\infty$. Although $f_k$ and $g_k$ need not be orthogonal,
a Pythagoras identity holds for this pair of functions obtained from $g_{k-1}$.

\begin{prop}
Let $t>0$,  $T_t$ and $S_t$ associated with $w,a$ as above, and $f \in \ell^2(V)$, then for each $k \in \mathbb{N}$,
if $f_k$ and $g_k$ are chosen as in the preceding definition, the Pythagoras identity
$$
    \|f_k\|^2 + \|g_k\|^2 = \|g_{k-1}\|^2 \, 
$$ 
holds.
\end{prop}
\begin{proof}
We verify that for $g_0=f$,
$\|f_1\|^2 = \|T_t  |f| \|^2 = \langle e^{-t \Delta_{w,a} }  |f|,  |f| \rangle$
while 
$ \|g_1\|^2 = \|S_t |f| \|^2 = \langle (I - e^{-t \Delta_{w,a} }) |f|, |f|\rangle \, ,
$
hence $\|f\|^2 = \| |f| \|^2 = \|f_1\|^2 + \|g_1\|^2$.
Now replacing $g_0$ by $g_{k-1}$ and $f_1$ and $g_1$ by $f_k$ and $g_k$, respectively, 
the same arithmetic gives the claimed identity.
\end{proof}

Together with the positivity preserving property, we can estimate the norm of each $g_k$ further,
in a type of sequence of Beurling-Deny inequalities.
\begin{prop}
Let $t>0$, $T_t$, $S_t$, $w$ and $a$ as in the preceding definition, then for $k \in \mathbb{N}$,
$$
  \| g_k \| = \| S_t |g_{k-1} | \| \le \| S_t g_{k-1} \|  \, .
$$
\end{prop}
\begin{proof}
To see this inequality, we note that the positivity preserving property implies
$
  \| T_t |g_k| \|^2 \ge \| T_t g_k \|^2 \, .
$
Using this in conjunction with the Pythagoras identity yields
$$
   \| S_t |g_k | \|^2 = \| g_k \|^2 - \| T_t |g_k| \|^2
   \le \|g_k \|^2 - \| T_t g_k \|^2 = \|S_t g_k \|^2 \, .
$$
Taking the square root gives the desired inequality.
\end{proof}

The Beurling-Deny inequality for $S_t$ permits us to remove the nonlinearity 
in the iteration defining the scattering transform, which provides estimates for the norm of $g_k$,
a first main result.
%We conclude with a main result on norm decay.

\begin{thm} Let the graph $(V,E)$ be weakly connected.
Let $w,a$ and $T_t$, $S_t$ be given as above, let $\lambda_{\mathrm{max}}$ be the largest eigenvalue of $\Delta_{w,a}$ and $f \in \ell^2(V)$ give rise to $(f_k,g_k)_{k=0}^\infty$
as described, then for each $k \in \mathbb N$,
$$
  \| g_{k} \| = \| S_t |g_{k-1}| \| \le (1-e^{-t \lambda_{\mathrm{max}}})^{(k-1)/2} \| S_t |g_0| \| \le 
   (1-e^{-t \lambda_{\mathrm{max}}})^{k/2} \| f \| \, . 
$$
\end{thm}
\begin{proof}
The estimate follows from the pair of inequalities
$$
  \| S_t |g_{k}| \| \le \| S_t g_k \| = \| S_t^2 |g_{k-1} | \|
  $$
  and 
  $$
    \| S_t^2 |g_{k-1} | \|^2 = \langle (I-T_t^2) S_t |g_{k-1} | , S_t |g_{k-1} |\rangle
    \le  (1-e^{-t \lambda_{\mathrm{max}}}) \| S_t  |g_{k-1}| \|^2 \, ,
  $$
  via the spectral theorem and monotonicity, which combines to
  $$ \| S_t |g_{k}| \| \le (1-e^{-t \lambda_{\mathrm{max}}})^{1/2} \| S_t  |g_{k-1}| \| \, .$$
  Iterating this
  in conjunction with $\|S_t |g_0| \| \le \| S_t g_0 \| \le \|g_0\| = \|f\|$
  gives the claimed norm estimate.
\end{proof}

Because of the decay of the norm and the Pythagoras identity, we obtain the following 
summation identity.

\begin{cor} Under the assumptions of the preceding theorem, for $f \in \ell^2(V)$,
we obtain
$$
  \|f\|^2 = \sum_{k=1}^\infty \|g_k\|^2
$$
and the terms in this series are dominated by a geometric series.
\end{cor}

The decay of the norm of $g_k$ for $k \to \infty$ is the justification that in practice
only a limited number of iterations are considered in the scattering transform.

We continue with some perturbation theory with the additional assumption that $a$ is a gradient field.
To set the stage, we compare the quadratic form of $S_t^2$ with that of $\Delta_{w,a}$ and derive a norm bound
for $g_1$.

\begin{lem} Let $(V,E)$ be as above, $w$ and $a$ and $\Delta_{w,a}= D_a^* M_w D_a$ as well as $S_t$ as defined before.
For $f \in \ell^2(V)$, we have
$$
    \| S_t f \| \le \sqrt t \| M_{\sqrt{w}} D_a f \| \, .
$$
\end{lem}
\begin{proof}
By the definition of $S_t$, $\| S_t f \|^2 = \langle S_t^2 f, f\rangle = \langle (I-e^{-t \Delta_{a,w}}) f, f\rangle$.
Now using the operator inequality $I - e^{-t \Delta_{w,a}} \le t \Delta_{w,a}$,
identifying the resulting quadratic form as $\langle \Delta_{a,w} f, f \rangle = \|M_{\sqrt w} D_a f\|^2$, and 
 taking the square root,
 gives the claimed bound.  
\end{proof}

This result permits us to modify the norm bound for $g_1$, which feeds into the successive applications
of $S_t$ and the nonlinear step in the scattering transform.

\begin{prop}
Let $(V,E)$ be as above, $w$ and $a$ and $\Delta_{w,a}= D_a^* M_w D_a$ as well as $S_t$, $f \in \ell^2(V)$
and $(f_k,g_k)_{k = 0}^\infty$ as described, then for $k \in \mathbb N$,
$$\| g_{k} \|  \le \sqrt t (1-e^{-t \lambda_{\mathrm{max}}})^{(k-1)/2} \| M_{\sqrt w} D_a  f  \|. $$
\end{prop}
\begin{proof}
We start from the inequality 
$$
  \| g_{k} \| = \| S_t |g_{k-1}| \| \le (1-e^{-t \lambda_{\mathrm{max}}})^{(k-1)/2} \| S_t |g_0| \|
$$
from the  preceding theorem, apply the Beurling-Deny inequality
$ \| S_t |g_0| \| \le \|S_t f \|$ and then the preceding lemma.
\end{proof}

As a consequence, if $a$ is a gradient field with potential $\phi$ and $f$ is close to a multiple of $e^\phi$, then the norm of each $g_k$ is suppressed.

\begin{cor}
Under the assumptions as in the preceding proposition, and if in addition $a$ is a gradient field with potential $\phi$
and $f= c e^{\phi} + h$ with $c \in \mathbb R$ and $h \in \ell^2(V)$, then for $k \in \mathbb N$,
$$
   \| g_{k} \|  \le \sqrt t (1-e^{-t \lambda_{\mathrm{max}}})^{(k-1)/2}     \| M_{\sqrt w} D_a  h  \|.
$$
\end{cor}

\section{Statistical properties of the scattering transform}\label{sec:stat}

In this section we make the additional assumption that the vertex set $V$ is equipped with a total ordering $<$,
so for any undirected edge $\{i,j\}$, we can assign an orientation such that $(i,j) \in E$ if $i<j$ and otherwise
$(j,i) \in E$. We then say that the graph becomes a directed graph induced by the total order on $V$.
For notational convenience, we also choose $V= \{1,2, \dots, n\}$ and use the order in $\mathbb N$.
%
%pass to the directed graph for which
%for any $(i,j) \in E$, $i<j$, and   

Next, we make a statistical assumption. Let $f$ be a stochastic process on $V$ with strictly positive values.
We assume that $f(i)=e^{\phi(i) +\nu(i)}$, where $\phi$ is deterministic and $\nu$ sub-exponentially distributed random noise.

\begin{defn}\label{def:multrandwalk}
Let $(V,E)$ be a directed graph whose edges are directed in accordance with the natural order on $V= \{1,2,\dots, n\}\subset \mathbb N$.
If a probability space $(\Omega,\mathcal F, \mathbb P)$ carries a filtration of $\sigma$-algebras $\{\mathcal{F}_i\}_{i \in V}$ indexed by
$V$ such that each $\nu(i)$ is measurable with respect to $\mathcal{F}_i$, each $e^{|\nu(i)|}$ is integrable with respect to $\mathbb P$, and for $j>i$, 
the increment $\nu(j) - \nu(i)$ is independent of $\mathcal{F}_i$ and 
the 
expectation of the multiplicative increment is $\mathbb E[ e^{\nu(j) - \nu(i)} ] = 1 $, then we say that
$\{e^{\nu(i)}\}_{i\in V}$ is a multiplicative random walk indexed by $V$.
\end{defn}

For convenience, we think of the random walk as initialized by an additional value  $\nu(0)=0$, although $0 \not \in V$.
% and let $V \subset \mathbb N$. 
With this notation, we have that for each $i \in \mathbb V$,
$$
    \mathbb E[ e^{\nu(i)} ] = \mathbb E[ e^{\nu(i)-\nu(i-1)} e^{\nu(i-1) - \nu(i-2)} \dots e^{\nu(1)-\nu(0)}  ] = 
    1
$$
by factoring the expected value. Consequently, $\mathbb E [ e^{\phi(i) + \nu(i)} ] = e^{\phi(i)}$,
so $e^{\nu(i)}$ can be thought of as multiplicative noise.

Next, we compute the first two moments of $D_a f$. The first result shows that under the model assumption we make,
the expected value of $f$ is a vector in the kernel of $\Delta_{w,a}$. 
%Moreover, we also assume that if there are multiple edges $\{(i,j_1), (i, j_2), \dots, (i, j_r)\}$ with initial vertex $i$, 
%then $\{\nu(j_1)-\nu(i), \nu(j_2) - \nu(i), \dots, \nu(j_r)- \nu(i)\}$ form an independent set of random variables.

\begin{lemma} If $a$ is a gradient field with potential $\phi$ and $\nu$ the exponent of a multiplicative random walk indexed by $V$ as in Definition~\ref{def:multrandwalk}, 
and $f=e^{\phi+\nu}$,
then for each $i \in V$,
$$
   \mathbb E [ \Delta_{w,a} f (i)]  = 0 \, .
$$
\end{lemma}
\begin{proof}
Under the above assumptions,
  $ \mathbb{E}[f(j) / f(i) ] = e^{\phi(j)-\phi(i)} \, .  $
   We deduce that if $a$ is the gradient field with potential $\phi$,
   $$D_a f(i,j) = f(j)- e^{a(i,j)} f(i)= e^{\phi(j)} (e^{\nu(j)}-e^{\nu(i)})  $$
   %The expected value is then
   %$$
    %  \mathbb E [ D_a f(i,j) ] = e^{\phi(j)} \mathbb E [ e^{\nu(j)} - e^{\nu(i)} ] \, .
   %$$
   Now using that 
    the increment $\nu(j) - \nu(i)$ is independent of $\nu(i)$, the expected value is then
   $$
      \mathbb E[(D_a f)(i,j)] = e^{\phi(j)} \mathbb E[ e^{\nu(i)}]  \mathbb E[e^{\nu(j) - \nu(i)}-1] \mathbb =  0 \, .
   $$
   The last identity holds because $\mathbb E[e^{\nu(j) - \nu(i)} ]=1$.
   By linearity, the expected value vanishes after composing with $M_w$ and with $D_a^*$.
   \end{proof}

Next, we wish to exploit the freedom to choose the weight $w$ appropriately.

\begin{lemma}
Let $f=e^{\phi+\nu}$ and assume $\nu$ has the properties as described in Definition~\ref{def:multrandwalk}, then
$$
    \mathbb{E} [ |(D_a f)(i,j)|^2 ]  = e^{2 \phi(j) }\mathbb E [ e^{2\nu(j)}  - e^{2\nu(i)} ] \, ,
$$ 
and setting $w(i,j) = (\mathbb{E} [ |(D_a f)(i,j)|^2 ])^{-1}$ then gives
$$
   \mathbb{E} [  |(M_{\sqrt w}D_a f)(i,j)|^2 ] = 1 \, .
$$
\end{lemma}
\begin{proof}
We compute the variance
\begin{align*}
  \mathbb{E} [ |(D_a f)(i,j)|^2 ]  & = e^{2 \phi(j) }\mathbb E [ (e^{\nu(j)}  - e^{\nu(i)})^2 ] \\ 
  &= e^{2 \phi(j) } \mathbb E [ (e^{\nu(j)-\nu(i)}  - 1)^2 ] \mathbb E[e^{2 \nu(i)}] \, .
\end{align*}
Now expanding the square, using that $\mathbb E[ e^{\nu(j)-\nu(i)}] = 1$, and canceling terms, we obtain the desired expression.
\end{proof}

We would like to make use of a measure concentration argument in order to bound the value
of the quadratic form associated with $\Delta_{w,a}$, and hence, of the norms in the scattering transform.
Typically, such arguments rely on sums of independent, identically distributed random variables.
In our setting, the assumption on the independence of the increment $\nu(j) - \nu(i)$ and $\mathcal{F}_i$ 
may still leave a set of dependent increments along the edges in $E$. To avoid technical complications,
we consider a Laplacian that is defined on a subset of edges.

\begin{defn}
Let $F \subset E$. We define the graph induced by the subset of edges $F$ to be the graph whose edge set
is $F$ and whose vertex set is $\partial F=\cup\{i \in e: e \in F\}$, the union of all vertices that are in edges contained in $F$.
A similar definition applies to directed graphs.
Given $w$ and $a$ as before, we define $\Delta_{w,a}^F$ to be the Laplacian on $\ell^2(\partial F)$ corresponding
to the quadratic form $Q_{w,a}^F(f) = \sum_{(i,j) \in F} w(i,j)(e^{a(i,j)} f(i) - f(j))^2$.
\end{defn}

\begin{prop}
Let $F \subset E$.
If $a$ is a gradient field with potential $\phi$, and $f = e^{\phi+\nu}$, $\nu$ as described in Definition~\ref{def:multrandwalk}, and
$w$ is chosen as in the preceding lemma, then
$$
    Q_{w,a}^F(f) = \sum_{(i,j) \in F} M_w(i,j) (D_a f(i,j))^2
$$  
is a random variable with expected value $|F|$.\end{prop}
\begin{proof}
This is a result of the linearity of the expectation and summing the terms in the preceding lemma.
\end{proof}

%Assuming $\phi$ is given, and we also know the second moment
%$$
% (w(i,j))^{-1} = \mathbb{E} [ |(D_a f)(i,j)|^2 ]  = e^{2 \phi(j) }\mathbb E [ (e^{\nu(j)}  - e^{\nu(i)})^2 ] \, .
%$$
%then each $M_{w} (D_{a} f)^2 (i,j)$ has expected value one and is independent of $\mathcal{F}_i$.

%We note that the decay in the norm of $g_k$ is controlled by the exponent $t \lambda_1$, which
%does not explicitly depend on the size of the graph. For example, for the hypercube and the standard graph Laplacian,
%$\lambda_1 $ is independent of the dimension of the hypercube, or equivalently, the number of vertices of the hypercube.

%Iterating this inequality provides an estimate between norms of the sequence
%$(g_k)$ and a sequence of high-pass components without application of the non-singular
%step.
%\begin{prop}
%Let $t>0$, $T_t$, $S_t$, $w$ and $a$ as in the preceding definition, then for $k \in \mathbb{N}$,
%$$
%  \| g_k \| = \| S_t |g_{k-1} | \| \le \| S_t g_{k-1} \|  = \| S_t S_t | g_{k-2} | \| \, .
%$$
%\end{prop}

By reducing $E$ to a subset $F$, we may choose 
it so that it contains edges for which the associated increments of $\nu$ are independent.
In addition, in order to get explicit bounds, we assume that each $\nu(i)$ has a normal distribution.

\begin{defn}
If the probability space $(\Omega, \mathcal F, \mathbb P)$, the directed graph $(V,E)$ 
and the stochastic process $\nu$ indexed by $V$ are as in Definition~\ref{def:multrandwalk}, and for each $i \in V$, $\nu(i)$ has a normal distribution,
then we say that $\{e^{\nu(i)}\}_{i \in V}$ is a log-normal multiplicative random walk indexed by $V$.
\end{defn}

We collect some properties of the stochastic process for this special case.

\begin{lemma}
If $\{e^{\nu(i)}\}_{i \in V}$ is a log-normal multiplicative random walk indexed by $V$, 
and $\mu_i = \mathbb E[\nu(i)]$, $\sigma_i^2 = \mathbb E[ \nu(i)^2 ] - \mu_i^2$, then
$\mu_i = - \frac{\sigma_i^2}{2}$.
\end{lemma}
\begin{proof}
This follows from the condition on the expected value $\mathbb E[ e^{\nu(i)} ] = e^{\mu_i + \sigma_i^2/2} = 1$. 
\end{proof}

\begin{lemma}
If $\{e^{\nu(i)}\}_{i \in V}$ is a log-normal multiplicative random walk indexed by $V$, 
and $\mu_i = \mathbb E[\nu(i)]$, $\sigma_i^2 = \mathbb E[ \nu(i)^2 ] - \mu_i^2$,
then the $n$-th moment is computed as 
$$ 
   \mathbb E [ e^{n \nu(i)} ] = e^{n(n-1) \sigma_i^2/2 } \, .
$$
\end{lemma}
\begin{proof}
This identity is a consequence of the expression for the $n$-th moment of a log-normal random variable,
$$
  \mathbb E [ e^{n \nu(i)} ] = e^{ n \mu_i + n^2 \sigma_i^2 / 2 } \, ,
$$
in conjunction with the relationship between $\mu_i$ and $\sigma_i^2$.
\end{proof}

\begin{cor}
If $\nu$ is a log-normal multiplicative random walk on $V$, $(V,E)$ 
a directed graph whose edges are aligned with the total ordering on $V$,
$w$ and $a$ are functions on $E$
as described, $a$ a gradient field with potential $\phi$, and $f=e^{\phi+\nu}$, then
we have for each $(i,j) \in E$, 
$$
   | (M_{\sqrt w} D_a f)(i,j) |^2 = \frac{(e^{\nu(j)} - e^{\nu(i)})^2 }{e^{\sigma_j^2} - e^{\sigma_i^2} }\, .
$$
\end{cor}

There are tail bounds on sums and differences of log-normal random variables, but they are
relatively technical and not explicit. We pursue a simpler approach and bound the variance of
the quadratic form $\langle \Delta_{w,a} f,f\rangle$. This means we need to compute 
higher moments of the log-normal random variables.

%In order to get a concentration result, we need a higher moment.

\begin{cor}
  With the same assumptions as in the preceding corollary, we have for each $(i,j) \in E$,
  $$
     \mathbb E [ | (M_{\sqrt w} D_a f)(i,j) |^4 ] 
     = e^{4 \sigma_j^2} + 2 e^{3 \sigma_j^2 + \sigma_i^2} + 3 e^{2 \sigma_j^2 + 2 \sigma_i^2} -3 e^{4 \sigma_i^2} \, .
  $$
\end{cor}
\begin{proof}
Using independence gives
$$
 \mathbb E [ | (M_{\sqrt w} D_a f)(i,j) |^4 ]  = (e^{\sigma_j^2} - e^{\sigma_i^2})^{-2} \mathbb E[ (e^{\nu(j)-\nu(i)} - 1)^4]
           \mathbb E[ e^{4 \nu(i)}]
$$
now inserting the expressions for the moments yields
$$
    \mathbb E [ | (M_{\sqrt w} D_a f)(i,j) |^4 ]  = (e^{\sigma_j^2} - e^{\sigma_i^2})^{-2}
    (e^{6 \sigma_j^2} - 4 e^{3 \sigma_j^2 + 3 \sigma_i^2} + 6 e^{\sigma_j^2 + 5 \sigma_i^2} - 3 e^{6\sigma_i^2} ) \, .
$$ 
Factoring the numerator and canceling terms gives the claimed expression.
\end{proof}

Assuming a path consisting of adjacent edges whose union is $F$,
%sequence $i_1, i_2, \dots, i_n$ such that $(i_j,i_{j+1}) \in E$ for each $j$, and 
using the independence assumption
for $\nu$ gives an expression for the variance of $ \langle \Delta_{w,a}^F f,f\rangle $, which we estimate.%.
%, $F = \{ (i_j,i_{j+1}) \}$.

\begin{thm}
Let $i_0, i_1, \dots, i_n$ be such that $i_j<i_{j+1}$, $(i_j,i_{j+1}) \in E$ for each $0 \le j \le n-1$, and assume $\nu$ is a log-normal
multiplicative random walk indexed by $V$, then the variance of the quadratic form associated with the Laplacian on $F = \{ (i_j,i_{j+1}):
0 \le j \le n-1 \}$ is bounded by
$$
      \mathbb E[(\langle \Delta_{w,a}^F f,f\rangle)^2] - n \le e^{4 \sigma_{i_n}^2} + \sum_{j=1}^n(3 e^{4 \sigma_{i_j}^2} - 1) - 3 \, .$$
\end{thm}
\begin{proof}
We renumber without loss of generality so that $i_j = j $.
When summing, there is a partial cancellation of terms as in a telescoping series,
$$
  \mathbb E[(\langle \Delta_{w,a}^F f,f\rangle)^2] - n = e^{4 \sigma_n^2} + \sum_{j=1}^n (2 e^{3 \sigma_j^2 + \sigma_i^2} + 3 e^{2 \sigma_j^2 + 2 \sigma_i^2}-1 -2 e^{4 \sigma_j^2}) - 3 e^{4 \sigma_0^2} \, .
$$
Now using the fact that for $j>i$, $\sigma_j \ge \sigma_i$, we estimate
$$
   \mathbb E[(\langle \Delta_{w,a}^F f,f\rangle)^2] - n \le e^{4 \sigma_n^2} + \sum_{j=1}^n ( 3 e^{4 \sigma_j^2} - 1) -3 \, .
$$
\end{proof}

We notice that the numerical constants in this bound cannot be improved further because if $n=1$ and $\sigma_0, \sigma_1 \to 0$, then the variance vanishes.

We abbreviate the bound
$$
   U = e^{4 \sigma_n^2} + \sum_{j=1}^n ( 3 e^{4 \sigma_j^2} - 1) -3 
$$
and formulate a corollary for the concentration of the quadratic form of $\Delta_{w,a}^F$.

\begin{cor}
Let $\phi, \nu$ be as above, $f = e^\phi+\nu$, $F \subset E$ consist of a sequence of adjacent edges whose 
vertices form an increasing sequence,
and consider the random variable $S_F=\langle \Delta^F_{w,a} f, f\rangle $, then for $\delta>0$,
$$
  \mathbb P [ S_F \ge |F| + \delta ] \le \frac{U}{U+\delta^2} \, .
$$
\end{cor}
\begin{proof}
This is a consequence of Cantelli's inequality \cite{cantelli} applied to the variance of $S_F$.
\end{proof}

In conjunction with the estimates for the norms in the scattering transform, we obtain more bounds.
\begin{cor}
Let $\phi, \nu$ be as above, $f = e^{\phi+\nu}$, $F \subset E$ consist of a sequence of adjacent edges whose 
vertices form an increasing sequence, let $(f_k,g_k)_{k=0}^n$ be defined with the scattering transform based
on $\Delta^F_{w,a}$
and consider the random variable $S_F=\langle \Delta^F_{w,a} f, f\rangle $, then for $\delta>0$,
$$
  \mathbb P [ \cup_{k=1}^n \{\|g_k\|^2 \ge t (1-e^{-t \lambda_{max}})^{(k-1)} 
  (|F| + \delta)\} ] \le \frac{U}{U+\delta^2} \, .
$$
\end{cor}

This implies that large norms in the scattering transform form a rare event. Consequently, when observing such a large value,
then we can with confidence reject the hypothesis that the distribution $e^{\phi(i)+\nu(i)}$ for $i \in F$
is in accordance with the chosen values for $\phi$ and $\sigma$.

\section{An application to anomaly detection in traffic counts}\label{sec:traffic}

In the following, we illustrate the results we discussed by applying the scattering transform to traffic count data \cite{CalTrans}. We wish to train the scattering transform so that deviations from the usual daily and weekly periodicities in traffic patterns can be detected.

The graph we consider is given by $V=\{(i_1, i_2): 1 \le i_1 \le 288, 1 \le i_2 \le 364\}$,
where the first coordinate of a vertex numbers consective time intervals of 5 minute length during each day
and the second coordinate is the number of a day in the year. We have taken the liberty to shorten the year to have 364 days
so that a year has exactly 52 weeks. The time intervals are chronologically ordered, so we have a total ordering on $V$.

The main model assumptions we make is that traffic counts for each day $d$ are given by a log-normal random walk
on $V_d= \{(i_1,i_2) \in V: i_2 = d\}$, whose increments are determined by $\phi$ and $\sigma$ as described in the preceding section.
We also assume that the realizations of $\nu$ on different days are independent.
Finally, we assume that the parameters $\phi$ and $\sigma$ only depend on which workday of the week is present,
for $1\le d\le 7$, $1 \le m \le 288$, $\{f(d+7j,m): 0 \le j \le 51\}$ forms a set of  independent and identically distributed random variables.   
Informally, each Monday the stochastic process is an independent realization governed by a choice of $\phi$ and  $\sigma$,
and each weekday is modeled in such a way.

As stated, the main objective in this proof-of-concept application is to find anomalies in the weekly periodicities of the traffic patterns given by holidays. Early attempts to use standard methods such as principle component analysis
and linear discriminants that are saturated with a given confidence level showed that just estimating $\phi$ and $\sigma$ is not sufficient. Even when a simple type of clustering was used to separate workdays from weekends,
there was no clear way to relate the outliers to known sources of anomalous traffic such as holidays.
A more sophisticated approach was needed.

To construct the scattering transform, we first define the set of directed edges.
This set encodes which values of $f$ are expected to be closest. 
We choose $(i, j) \in E$ if $j_1-i_1=1$ and $j_2=i_2$ or if $j_1=i_1$ and $j_2 = i_2+7$. This means, two vertices are adjacent 
if they belong to two consecutive time blocks or if they represent the same time block spaced by one or two weeks.

For the implementation of the scattering transform, we restrict the edges to subsets. For each day $d$, the subset $F_d$ of the edge set given by the edges between the time blocks on that given day
and the edges in $E$ between vertices in this day and the vertices in the preceding or following two weeks at the same time, 
{\small \begin{align*}
   F_d = \{ (i, j) \in E: \, & i_1= j_1, i_2 = d, j_2=d+7 , \mbox{ or }  i_1 = j_1, i_2 = d+7, j_2 = d+14,\\
& \mbox{ or } i_1 = j_1, i_2 = d-7, j_2 = d,
   \mbox{ or } i_1 = j_1, i_2 = d-14, j_2 = d-7,     \\ & \mbox{ or } j_1 = i_1+1, i_2 = j_2 = d\} \, .
\end{align*}}
\noindent We have chosen a 5-week window for $F_d$ in order to have enough neighboring vertices so that anomalous traffic counts can dissipate under the heat semigroup.

An example of an occurrence of anomalous traffic is presented in Figure~\ref{fig:labor}. There, the most obvious response of the network is the output of the first layer. The traffic counts through the year and the 
resulting response of the first layer of the scattering transform are pictured in Figure~\ref{fig:layer}.

% \begin{figure} 
%\hspace*{-.6cm}
% \includegraphics[width=1.8in]{GraphScatterPics/scdl148-2.pdf}\hspace*{-.1cm}%
% \includegraphics[width=1.8in]{GraphScatterPics/scdl149-2.pdf}\hspace*{-.1cm}%
% \includegraphics[width=1.8in]{GraphScatterPics/scdl150-2.pdf}
% \caption{Memorial day weekend, from left to right: Sunday, Monday (holiday), Tuesday (workday)}
% \end{figure}
 
  \begin{figure}
 \hspace*{-.6cm} \includegraphics[width=1.8in]{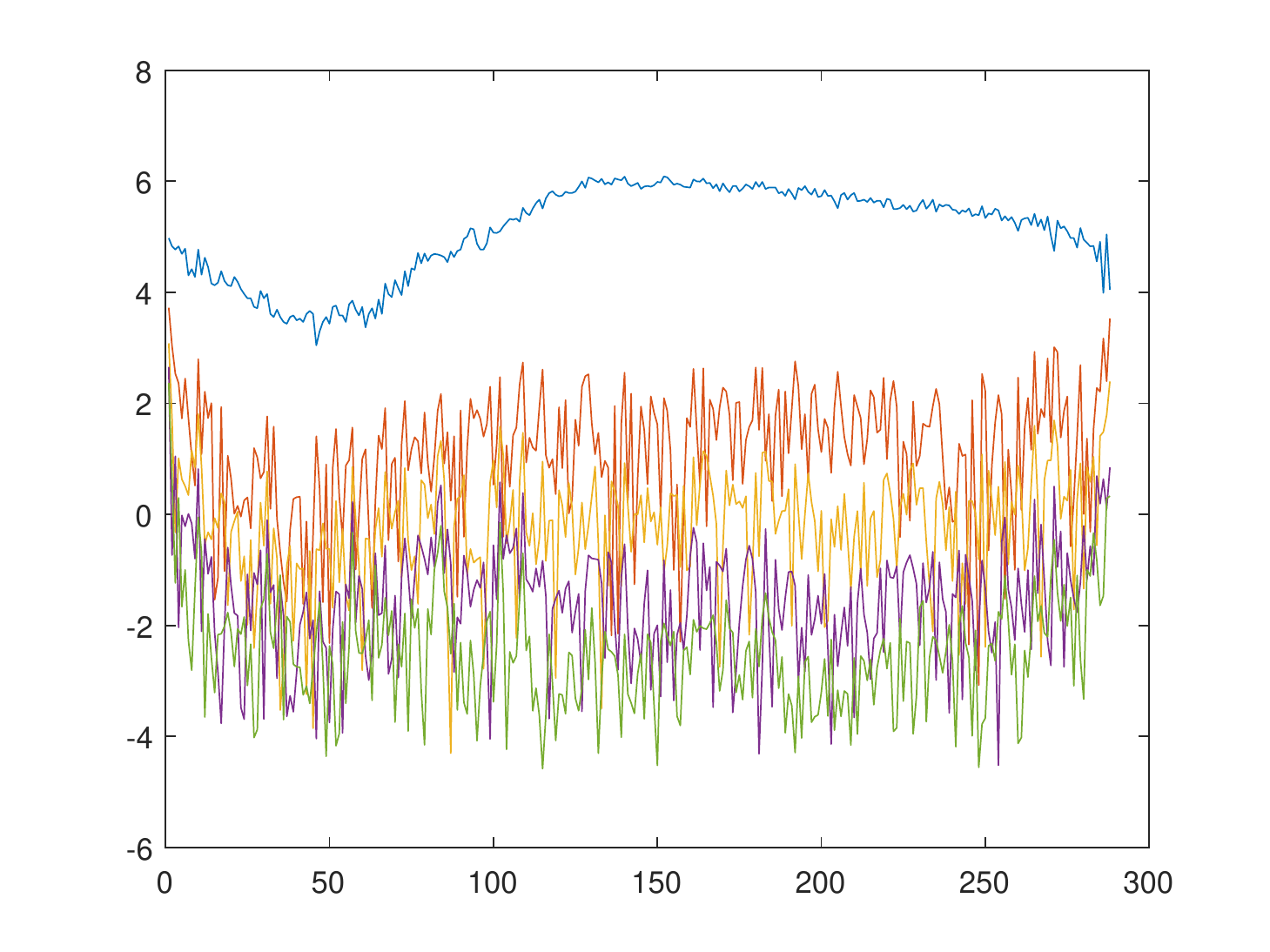}\hspace*{-.1cm}%
 \includegraphics[width=1.8in]{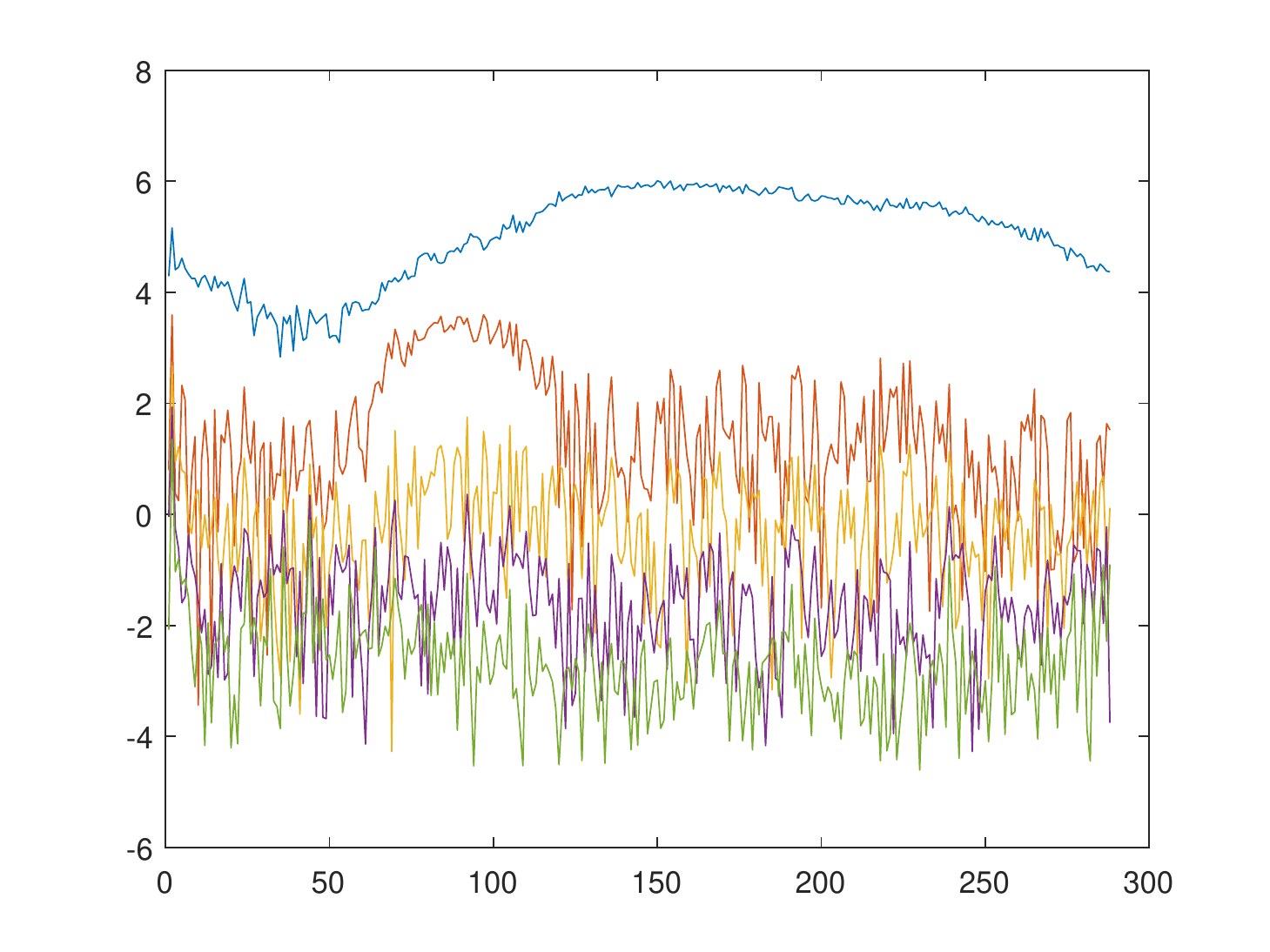}\hspace*{-.1cm}%
 \includegraphics[width=1.8in]{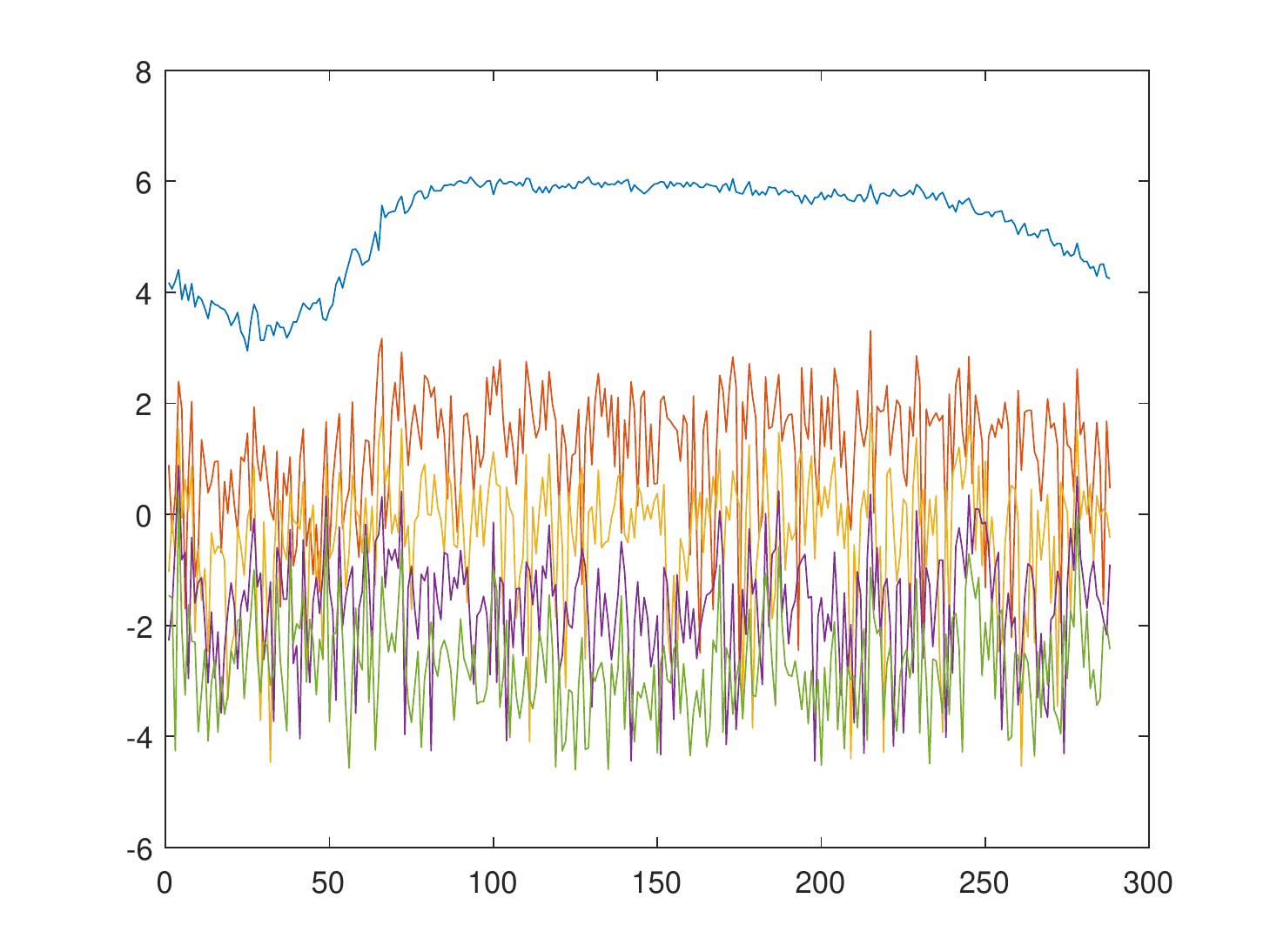}
 \caption{Labor Day weekend, from left to right: Sunday, Monday (holiday), Tuesday (workday). The plots represent the logarithm of the traffic counts (top) or
 of the output of the scattering transform at one station during the 288 consecutive 5-minute intervals in one day. The output of the first layer (red) responds to a deviation from the usual, weekly periodicity in the traffic intensities on Labor Day Monday. }
 \label{fig:labor}
 \end{figure}
 
 \begin{figure} 
\hspace*{-.6cm}
 \includegraphics[width=2.8in]{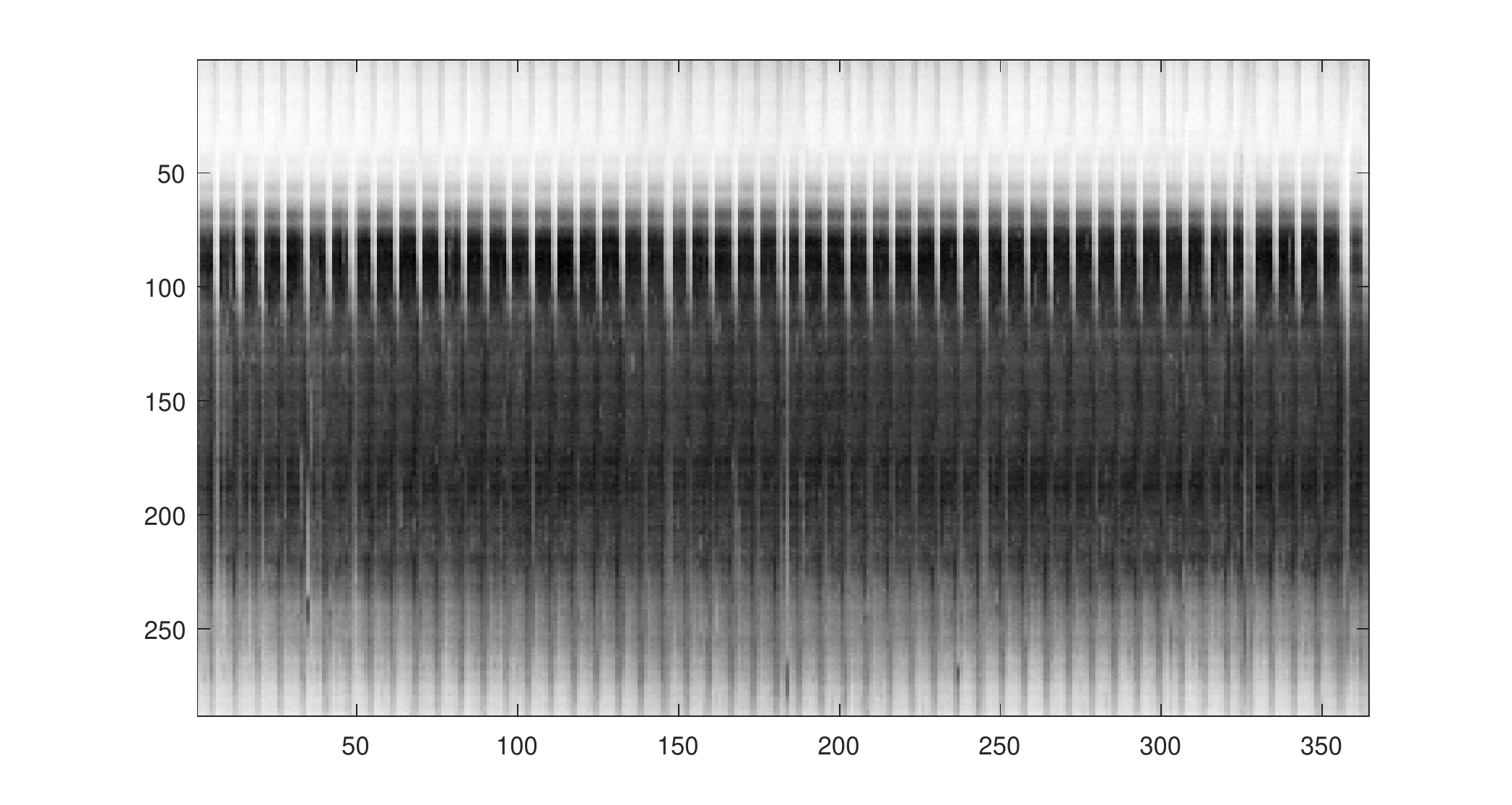}\hspace*{-.4cm}%
 \includegraphics[width=2.8in]{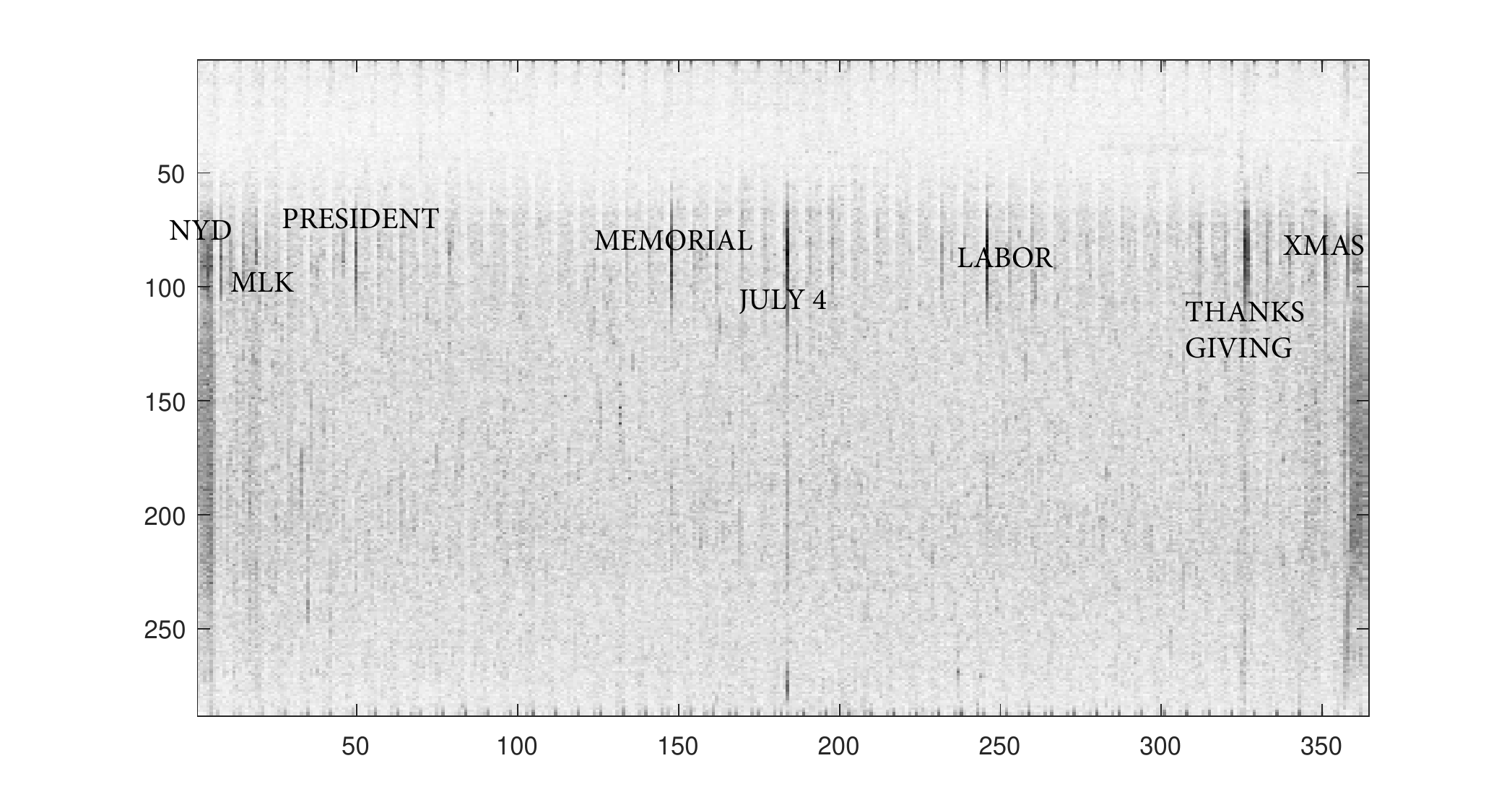}%\hspace*{-.1cm}%
 \caption{Traffic counts at a sensor in California during the year 2017 (left) and the output of the first layer of the scattering transform (right). The vertical coordinate denumerates the 288 5-minute blocks, the horizontal coordinate the day of the year. The grayscale is chosen so that dark represents higher values than light. The peaks in the output of the first layer are visible as dark vertical line segments. They represent traffic counts that diverge from the usual weekly periodicity. The reason for such significant anomalous behavior is most often a public holiday, as indicated in the image.}\label{fig:layer}
 \end{figure}

When computing the principal components of the output and deriving linear discriminants that can be used to
confine most of the output, linear combinations of the different layers appear, which show that some layers and some time periods are more important than others.
The outcome of the experiments
indicate that the norm estimates we derived could be refined with a more localized treatment in order to isolate shorter time intervals in which the traffic pattern deviates from the expected values. 
This will be pursued in future work.

%We know that asymptotically, $S$ approaches a chi-squared distributed random variable. In order to get more concrete expressions,
%we assume that each $\nu(j) - \nu(i)$ is normally distributed. The condition $\mathbb E[ e^{\nu(j) - \nu(i)} ]=1$
%then requires that $\mu(i,j) = \mathbb E[\nu(j) - \nu(i)]$ and $\sigma^2(i,j) = \mathbb E[ (\nu(j) - \nu(i) )^2 - (\mu(i,j))^2$
%satisfy $\mu(i,j) = -\sigma^2(i,j)/2$. Using the same technique as before, we then get
%$
%   \mathbb E[ \nu(j) ] = \mu_j
%$
%and 
%$
%  (\sigma(j))^2 = \mathbb E[ (\nu(j))^2 - \mu_j^2 ] 
%$
%satisfy
%$$
%   \mu_j = - \frac{(\sigma(j))^2}{2} \, .
%$$
%Now using the expression for the moments of a log-normal random variable gives
%$$
%  \mathbb E[ e^{2 \nu(i)} ] = e^{2 \mu(i) +2 (\sigma(i))^2} = e^{(\sigma(i))^2} \, 
%$$
%and
%$$
%\mathbb E [ (e^{\nu(j)-\nu(i)}  - 1)^2 ] = \mathbb E[ e^{2 \nu(j)-2 \nu(i)} - 1]
%= e^{\sigma^2(j) - \sigma^2(i)}-1 \, .
%$$
%

\end{document}